\documentclass[10pt]{article}
\usepackage{amsmath, amssymb,amsthm,tocloft, enumerate, tikz, algorithmicx,
algpseudocode, mathrsfs, tkz-graph, color, comment}
\usepackage[margin=1.5in]{geometry}
\usepackage[colorlinks]{hyperref}
\hypersetup{linkcolor=blue,citecolor=blue,urlcolor=black}

\newtheorem{theorem}{Theorem}[section]
\newtheorem{lemma}[theorem]{Lemma}
\newtheorem{proposition}[theorem]{Proposition}
\newtheorem{corollary}[theorem]{Corollary}

\theoremstyle{definition}
\newtheorem{definition}[theorem]{Definition}

\theoremstyle{remark}

\numberwithin{equation}{section}

\author{Casey Donoven\footnote{Email address: \href{mailto:cdonoven@binghamton.com}{\nolinkurl{cdonoven@binghamton.com}}}}
\title{Groups that are the union of two semigroups have left-orderable quotients}

\begin{document}
\maketitle
\begin{abstract}
In this article, we show that a group $G$ is the union of two proper subsemigroups if and only if $G$ has a nontrivial left-orderable quotient.  Furthermore, if $G$ is the union of two proper semigroups, then there exists a minimum normal subgroup $N\unlhd G$ for which $G/N$ is left-orderable and nontrivial.
\end{abstract}
\section{Introduction}
The \emph{covering number} of group $G$ with respect to subgroups, $\sigma_g(G)$, is the minimum number of proper subgroups of $G$ whose union is $G$.  The covering number of groups has been extensively studied and was formally defined by \cite{Cohn94}.  Early results on covering numbers (not phrased as such) include \cite{Scorza26}, in which Scorza showed that a group has covering number three if and only if $G$ has a homomorphic image isomorphic to the Klein-Four group.  While is it is elementary to show no group is the union of two proper subgroups, it is also the case that no group has covering number seven \cite{Tomkinson97}.  It is now known for all $n$ satisfying $2\leq n\leq 129$ whether $n$ is a covering number of a group \cite{GaronziKappeSwartz18}.  Similar studies have explored analogous results for rings and loops, see \cite{GagolaKappe16}, \cite{Lucchini12}, and \cite{Werner15}.  

This paper explores covering groups with subsemigroups, as opposed to subgroups.  A \emph{semigroup} is a set with an associative operation and a \emph{subsemigroup} is simply a subset of a semigroup that is closed with respect to the inherited opertaion. Note that all groups are semigroups, but semigroups need not have an identity or inverses.  The covering number of a semigroup $S$ with respect to subsemigroups, $\sigma_s(S)$, is defined analogously to covering numbers of groups.  Covering numbers of semigroups are explored in \cite{DonovenKappe} and are characterized for finite semigroups and some specific classes of semigroups.

While a group is never the union of two proper subgroups, a group may be the union of two proper subsemigroups.  For example, the additive group of integers, $\mathbb{Z}$, is the union of two proper subsemigroups, namely the positive and non-positive integers. Our main result characterizes precisely when a group is the union of two semigroups.  Before stating our main result, we first give the following definition of left-orderable groups and a proposition alluding to the relationship between left-orderable groups and semigroups.

\begin{definition}
A group $G$ is \emph{left-orderable} when there is a total order $\leq$ on $G$ that respects left multiplication, i.e.~for $g_1,g_2,h\in G$, we have $g_1\leq g_2$ if and only if $hg_1\leq hg_2$.
\end{definition}

Throughout this paper, we will use the following proposition as an equivalent definition of left-orderable groups.  For a subset $A$ of a group $G$, define $A^{-1}=\{a^{-1}\mid a\in A\}$.

\begin{proposition}\label{pro:loequivalentdefn}
If $G$ is a group with left-order $\leq$, then $P=\{g\in G\mid 1\leq g\}$ satisfies $P\cup P^{-1}=G$ and $P\cap P^{-1}=\{1\}$.  Conversely, if $G$ is a group with subsemigroup $P$ satisfying $P\cup P^{-1}=G$ and $P\cap P^{-1}=\{1\}$, then $G$ has a left-order $\leq$ defined by $g\leq h$ if and only if $g^{-1}h\in P$.
\end{proposition}

Examples of left-orderable groups include torsion-free abelian or nilpotent groups, free groups, and Thompson's group $F$.
See \cite{DeroinNavasRivas14} for more details and examples of left-orderable groups.

Our main result extends the relationship between left-orderability and groups as the union of two subsemigroups.

\begin{theorem}\label{thm:MainThm}
A group $G$ is the union of two proper subsemigroups if and only if $G$ has a nontrivial left-orderable quotient.
\end{theorem}

As a brief example, consider $G=\mathbb{Z}\times C_2$, where $C_2$ is the cyclic group of order two.  Since $G$ has elements of finite order, $G$ is not left-orderable.  However, $G$ is the union of two proper subsemigroups, $P\times C_2$ and $P^{-1}\times C_2$, where $P$ is the set of non-negative integers.  Also, it is clear that $G$ quotients onto $\mathbb{Z}$ and thus has a left-orderable quotient.

After we prove Theorem~\ref{thm:MainThm}, we give some simple remarks on minimality of normal subgroups inducing left-orderable quotients and finish with some open questions.

\section{Proof of Theorem~\ref{thm:MainThm}}
In this section, we give a proof of Theorem~\ref{thm:MainThm} after presenting several useful lemmas.  We begin with the proof of the reverse implication in Theorem~\ref{thm:MainThm}.  

\begin{proposition}
Let $G$ be a group and $H\unlhd G$ such that $G/H$ is left-orderable and not the trivial group.  Then $G$ is the union of two proper subsemigroups.
\end{proposition}
\begin{proof}
Since $G/H$ is a nontrivial left-orderable group, $G/H$ has a proper subsemigroup $P=\{gH\in G/H \mid H\leq gH\}$ where $\leq$ is the order on $G/H$.  Moreover, $P^{-1}$ is also a proper subsemigroup of $G/H$ such that $P\cup P^{-1}=G/H$.  Letting $\phi:G\to G/H$ be the quotient map, we see $\phi^{-1}(P)$ and $\phi^{-1}(P^{-1})$ are proper subsemigroups of $G$ such that $\phi^{-1}(P)\cup\phi^{-1}(P^{-1})=G$.
\end{proof}

For the remainder of this section, let $G$ be a group such that $G$ is the union of two proper subsemigroups, $A$ and $B$.  Note that if $S$ is a proper subsemigroup of $G$, then $S\cup \{1\}$ is also a proper subsemigroup, so we implicitly assume $1\in A\cap B$.  

Define $I=A\cap B$.  We use $\langle I \rangle$ to mean the group generated by $I$.  The following four lemmas will be used to show that we may assume $I=\{1\}$ without loss of generality.

\begin{lemma}\label{lem:intersectioninB}
$\langle I \rangle$ is contained in $A$ or $B$.
\end{lemma}
\begin{proof}
Consider the following two disjoint sets: $$I_A=\{x\in I \mid  x^{-1}\in A\text{ and }x^{-1}\not\in B\}$$ and $$I_B=\{y\in I \mid  y^{-1}\in B\text{ and }y^{-1}\not\in A\}.$$  Suppose that $x\in I_A$ and $y\in I_B$.  Then the element $x^{-1}y^{-1}$ must be in $A$ or $B$.  If $x^{-1}y^{-1}\in A$, then $xx^{-1}y^{-1}=y^{-1}\in A$ which contradicts $y\in I_B$. Likewise, if $x^{-1}y^{-1}\in B$, then $x^{-1}y^{-1}y=x^{-1}\in B$ which contradicts $x\in I_A$. Therefore $I_A$ or $I_B$ is empty.

Without loss of generality, assume $I_A=\emptyset$.  This implies that $I^{-1}\subseteq B$ and thus the group generated by $I$ is a subset of $B$.
\end{proof}

Henceforth, we will assume $\langle I \rangle\subseteq B$.  Note that the inverse of \emph{some} elements in $I$ may be contained in $A$, however the inverse of \emph{every} element in $I$ is contained in $B$. 

Define $H=\{h\in B \mid h^{-1}\in B\}$.  We see $H$ is a subgroup of $B$ and moreover $H$ is the maximal subgroup of $B$ with respect to inclusion.  Note that $\langle I\rangle \leq H$.

\begin{lemma}\label{lem:AtimesHsecond}
If $h\in H$, then $h(A-I)=(A-I)=(A-I)h$ and  $h(B-H)=(B-H)=(B-H)h$.
\end{lemma}
\begin{proof}
Since $B$ is a semigroup,  $bH\subseteq B$ for all $b\in B$. Thus, $B$ is a union of left cosets of $H$.  The complement of $B$, i.e.~$A-I$, is also a union of left cosets of $H$. This implies $(A-I)h=(A-I)$ and $(B-H)h=(B-H)$. A similar argument with right cosets finishes the proof.
\end{proof}

The following lemma describes the inverses of elements in $A$ and $B$.

\begin{lemma}\label{lem:AandBareinversessecond}
$(A-I)^{-1}=B-H$
\end{lemma}
\begin{proof}
Let $b\in B-H$.  Then $b^{-1}\not\in B$ and thus $b^{-1}\in A-I$.  Also let $a\in (A-I)$.  Suppose for contradiction that $a^{-1}\in A-I$. 

If there exists an $h\in H$ such that $h\not\in I$,  then $ha\in A$ by the previous lemma.  Therefore $haa^{-1}=h\in A$, which is a contradiction.  Therefore $a^{-1}\in B-H$.

However, if there does not exist an $h\in H$ such that $h\not\in I$, then $\langle I\rangle =H=I$.  Let $H'=\{h\in A \mid h^{-1}\in A\}$.  We see that for all $a'\in A-H'$, $a'\in B-H'=B-I$ and therefore $(B-I)^{-1}=A-H'$.  In this case, without loss of generality, switch the names of $A$ and $B$ as well as $H$ and $H'$ to complete the proof.
\end{proof}

Note that Lemma~\ref{lem:AandBareinversessecond} implies every subgroup of $A$ is contained in $I$.

\begin{lemma}\label{lem:A-Isemigroupsecond}
$A-I$ is a semigroup.
\end{lemma}
\begin{proof}
Let $a_1,a_2\in A-I$.  Assume for contradiction that $a_1a_2\in I$.  Since $I\subseteq H$, this means $a_1a_2\in H$ and therefore $a_1\in Ha_2^{-1}$.  By Lemma~\ref{lem:AandBareinversessecond}, $a_2^{-1}\in B-H$ so we see $Ha_2^{-1}\subseteq B-H$.  Since $B-H$ is the complement of $A$, this contradicts the fact that $a_1\in A$.
\end{proof}

Using Lemma~\ref{lem:A-Isemigroupsecond}, we see that $G$ is the union of two proper semigroups, $(A-I)\cup\{1\}$ and $B$, who intersect only on the identity. For the rest of the paper, we will assume without loss of generality that $I=A\cap B=\{1\}$.

As an aside, we point out that all torsion elements of $G$ must be contained in $H$, the maximal subgroup of $B$. We express a consequence of this in the following proposition.

\begin{proposition}\label{pro:torsionsemigroups}
A group that is generated by elements of finite order is not the union of two proper subsemigroups.
\end{proposition}
\begin{proof}
Let the group $G$ be the union of proper subsemigroups, $A$ and $B$, with the same assumptions on $A$ and $B$ as above.  If $g\in G$ has order $n$, then $g^{n-1}=g^{-1}$.  This implies $g,g^{-1}\not\in A-\{1\}$, since $A-\{1\}$ is closed under multiplication by Lemma~\ref{lem:A-Isemigroupsecond}, but does not contain the inverses of any of its elements by Lemma~\ref{lem:AandBareinversessecond}.  We see $g,g^{-1}\in B$ and thus $g\in H$.  We conclude that if $G$ were generated by elements of finite order, then $H$ contains a generating set of $G$ so $G=H$, which is a contradiction.
\end{proof}

With the assumptions on the subsemigroups $A$ and $B$, we can now construct left-orderable quotients of $G$.

If $H\unlhd G$, then Lemma~\ref{lem:AandBareinversessecond} implies that $B/H\cap(B/H)^{-1}= \{H\}$ and $ B/H \cup(B/H)^{-1}= G/H$. We then see $G/H$ is left-orderable using Proposition~\ref{pro:loequivalentdefn}, where the order is defined as $xH\leq yH$ if and only if $x^{-1}y\in B$

If $H\ntrianglelefteq G$, we construct new subsemigroups $A'$ and $B'$ whose union is $G$ that will be used to construct a left-order. Fix a $g\in G$ such that $H^g\neq H$.  Define $H_A=H\cap A^{g^{-1}}$ and $H_B=H\cap B^{g^{-1}}$.  Since $H^g\neq H$ and $H$ is the maximal subgroup of $B$, $H_A$ contains a non-identity element.  Also, $H_A$ is a semigroup since it is the intersection of two semigroups.  Similarly, $H_B$ is a semigroup.  Note that $H=H_A\cup H_B$ and $H_A\cap H_B=\{1\}$.

Define $A'=A\cup H_A$ and $B'=(B-H_A)\cup\{1\}$.

\begin{lemma}\label{lem:AprimeBprime}
$A'$ and $B'$ are semigroups.
\end{lemma}  
\begin{proof}
Since $A$ and $H_A$ are semigroups, Lemma~\ref{lem:AtimesHsecond} implies that $A'$ is a semigroup.

Let $b_1,b_2\in B'$.  Firstly, note that if $b_1,b_2\not\in B-H$, then $b_1,b_2\in H_B$ which means $b_1b_2\in H_B$.  Therefore $b_1b_2\in B'$.

Now suppose for contradiction that $b_1b_2=h\in H_A$, implying that either $b_1$ or $b_2$ must be contained in $B-H$.  If $b_1\in B-H$, then $b_1^{-1}\in A-\{1\}$.  This implies $b_2=b_1^{-1}h\in A-\{1\}$ by Lemma~\ref{lem:AtimesHsecond}, which is a contradiction.  We get a similar contradiction if $b_2\in B-H$. We can conclude that $b_1b_2\in B'$ and $B'$ is a semigroup.
\end{proof}

We have constructed a new pair of semigroups $A'$ and $B'$ such that $G=A'\cup B'$ and $A'\cap B'=\{1\}$.  It is also important to note that that $A\subsetneq A'$ and $B'\subsetneq B$.  To further the comparison between $A$ and $A'$ and $B$ and $B'$, the following lemma parallels Lemma~\ref{lem:AandBareinversessecond}.

\begin{lemma}\label{lem:Aprimeinverse}
If $a\in A'$, then $a^{-1}\in B'$.
\end{lemma}
\begin{proof}
If $a\in A-\{1\}$, then $a^{-1}\in B-H\subseteq B'$.

Now let $a\in H_A$ with $a\neq 1$.  Then $a^{-1}\in H$ since $H$ is a group. Suppose for contradiction that $a^{-1}\in H_A$.  Then $a^{-1}\in A^{g^{-1}}$, meaning $(a^{-1})^g\in A$ and therefore $(a^g)^{-1}\in A$.   However, by definition $a\in H_A$ implies $a^g\in A$. Having both $a^g,(a^g)^{-1}\in A$ contradicts Lemma~\ref{lem:AandBareinversessecond} so $a^{-1}\not\in H_A$. 
\end{proof}

Like $A$, the semigroup $A'$ does not contain a nontrivial subgroup.

We now consider the space $\mathcal{F}$ of pairs of proper subsemigroups of $G$, $(U,V)$, such that
\begin{enumerate}
\item $G=U\cup V$ and $U\cap V=\{1\}$;
\item $A\subseteq U$ and $V\subseteq B$;
\item $U$ does not contain a nontrivial subgroup.
\end{enumerate}

Define a partial order on this space as $(U_1,V_1)\leq (U_2,V_2)$ if and only if $V_1\subseteq V_2$. Note that $V_1\subseteq V_2$ if and only if $U_2\subseteq U_1$.

\begin{lemma}
$\mathcal{F}$ has a minimal element
\end{lemma}
\begin{proof}
Let $\{(U_i, V_i)\}_i$ be a chain in $\mathcal{F}$.   We claim that $(\bigcup U_i, \bigcap V_i)\in\mathcal{F}$.  Clearly $\bigcap V_i$ is a proper subsemigroup of $G$, as the intersection of semigroups is a semigroup.  Let $x,y\in \bigcup U_i$. There then exists an $n$ such that $x,y\in U_n$ and therefore $xy\in U_n$.  We see that $\bigcup U_i$ is also a semigroup. (We show it is proper later.)

For condition 1, clearly $(\bigcup U_i)\cup (\bigcap V_i)\subseteq G$.  For the reverse containment, let $g\in G$.  If there exists an $n$ such that $g\in U_n$, then $g\in \bigcup U_n$.  If there is no $n$ such that $g\in U_n$, then $g\in V_i$ for all $i$.  Therefore $g\in (\bigcap V_i)$.  In either case, $g\in (\bigcup U_i)\cup (\bigcap V_i)$.  Lastly, if $h\in (\bigcap V_i)$, then $h\not\in U_i$ for each $i$ unless $h=1$, implying $(\bigcup U_i)\cap (\bigcap V_i)=\{1\}.$

Condition 2 is straightforward using the fact that $A\subseteq U_i$ and $V_i\subseteq B$ for all $i$.  

Lastly, suppose $\bigcup U_i$ contains a nontrivial subgroup.  This would imply that there exists a $g\in \bigcup U_i$ such that $g^{-1}\in\bigcup U_i$.  Therefore there exists an $n$ such that $g,g^{-1}\in U_n$, contradicting the fact that $U_n$ has no nontrivial subgroups.  This also implies $\bigcup U_i$ is proper in $G$.

By Zorn's Lemma, $\mathcal{F}$ has a minimal element.
\end{proof}

This minimal element will give us a partial order on $G$.

\begin{lemma}\label{lem:MinimalPairOfSemigroups}
Let $(U,V)$ be a minimal element of $\mathcal{F}$.  Then the subgroup \linebreak $N=\{h\in V \mid h^{-1}\in V\}$ is normal in $G$. 
\end{lemma}
\begin{proof}
Assume for contradiction that $N\ntrianglelefteq G$.  Then there exists an element $g\in G$ such that $N^g\neq N$.  Using $g$, define $U'$ and $V'$ analogously to $A'$ and $B'$ before Lemma~\ref{lem:AprimeBprime}.  Then $(U',V')\in\mathcal{F}$ and $(U',V')\lneq(U,V)$, which is a contradiction.
\end{proof}

Since $N\unlhd G$, we see that $G/N$ is left-orderable, where the order is defined as $xN\leq yN$ if and only if $x^{-1}y\in B/N$.

We now give a proof of Theorem~\ref{thm:MainThm}. 

\begin{proof}[Proof of Theorem~\ref{thm:MainThm}]
Let $G$ be a group with proper subsemigroups $A$ and $B$ such that $G=A\cup B$.  We may assume without loss of generality that $A\cap B=\{1\}$ and $A^{-1}\subseteq B$, using Lemmas~\ref{lem:intersectioninB}, \ref{lem:AtimesHsecond}, \ref{lem:AandBareinversessecond}, and \ref{lem:A-Isemigroupsecond}.  Let $H$ be the maximal subgroup of $B$ with respect to inclusion.  If $H\unlhd G$, then $G/H$ is left-orderable where  $xH\leq yH$ if and only if $x^{-1}y\in B$.  If $H\ntrianglelefteq G$, let $\mathcal{F}$ be the partially ordered family of pair of subsemigroups defined above.  Further let $(U,V)$ be the minimal pair of subsemigroups with $N$ being the maximal subgroup of $V$.  Since $N\unlhd G$ by Lemma~\ref{lem:MinimalPairOfSemigroups}, we see that $G/N$ is left-orderable, where the order is defined as $xN\leq yN$ if and only if $x^{-1}y\in V$.
\end{proof} 

We can also state a corollary of Proposition~\ref{pro:torsionsemigroups}.

\begin{corollary}
A group generated by elements of finite order has no nontrivial left-orderable quotients.
\end{corollary}

\section{Minimal Normal Subgroups and Coverings}
In this section, we include some brief remarks on minimal normal subgroups inducing left-orderable quotients and coverings of groups by two proper subgroups.

\begin{proposition}\label{pro:IntersectingOrders}
Let $N_1,N_2\unlhd G$ such that $G/N_1$ and $G/N_2$ are left-orderable.  Then $G/(N_1\cap N_2)$ is left-orderable. Furthermore, both $N_1/(N_1\cap N_2)$ and $N_2/(N_1\cap N_2)$ are also left-orderable.
\end{proposition} 
\begin{proof}
Let $\leq_1$ be the order on $G/N_1$ and $\leq_2$ be the order on $G/N_2$.  Define a partial order $\leq$ on  $G/(N_1\cap N_2)$ as $a(N_1\cap N_2)\leq b(N_1\cap N_2)$ if and only if $a N_1< b N_1$ or $aN_1=bN_1$ and $aN_2\leq b N_2$.  It is clear that this is a left order on $G/(N_1\cap N_2)$ as both $\leq_1$ and $\leq_2$ are left orders.

Every subgroup of a left-orderable group is left-orderable, simply by restricting the order to the subgroup.  Therefore $N_1/(N_1\cap N_2)$ and $N_2/(N_1\cap N_2)$ are also left-orderable as they are subgroups of $G/(N_1\cap N_2)$.
\end{proof}

Proposition~\ref{pro:IntersectingOrders} indicates the presence of a minimal normal subgroup inducing a left-orderable quotient, which is simply the intersection of all normal subgroups inducing left-orderable quotients.  Also, given two covering of $G$ by two proper subsemigroups, we may construct a `new' covering from the order given in Proposition~\ref{pro:IntersectingOrders}. 

Let $G=A_1\cup B_1=A_2\cup B_2$ where $A_1,A_2, B_1,B_2$ are proper subsemigroups of $G$ with the usual assumptions that $A_1\cap B_1=\{1\}$ and $A_2\cap B_2=\{1\}$, $A_1$ and $A_2$ contain no nontrivial subgroups, and the maximal subgroups of $B_1$ and $B_2$ are normal in $G$.  (Essentially, we pass to a minimal element of the partially ordered pairs given by Lemma~\ref{lem:MinimalPairOfSemigroups}.) Let $N_1$ and $N_2$ be the maximal subgroups of $B_1$ and $B_2$ respectively, with the left orders on the quotients being $x\leq_i y$ if and only if $x^{-1}y\in B_i/N_i$.  

Define $B'=\{g\in G \mid N_1\cap N_2\leq g(N_1\cap N_2)\}$ where $\leq$ is the partial order given in the proof of Proposition~\ref{pro:IntersectingOrders}.  Notice that $B'$ is the preimage of the non-negative elements of $G/(N_1\cap N_2)$. We see that $B'$ is the union of the preimage of strictly positive elements with respect to $\leq_1$ (i.e.~$B_1-N_1$) and the elements of $N_1$ that are preimages of non-negative elements with respect to $\leq_2$ (i.e.~$N_1\cap B_2$).  This implies $B'=(B_1-N_1)\cup (N_1\cap B_2)$ and therefore $B'\subseteq B_1$. 

We may also define $A'=\{g\in G \mid g(N_1\cap N_2)< N_1\cap N_2\}\cup \{1\}$ and we see that $A'$ and $B'$ are proper subsemigroups of $G$ such that $G=A' \cup B'$, $A'\cap B'=\{1\}$, $A'$ contains no nontrivial subgroups, and the maximal subgroup of B', $N_1\cap N_2$, is normal in $G$.  Furthermore, $A_1\subseteq A'$ and $B'\subseteq B_1$.  

\section{Open Questions}
Recall the covering number of a group $G$ with respect to semigroups, $\sigma_s(G)$, is the minimum number of proper subsemigroups of $G$ whose union is $G$.  Theorem~\ref{thm:MainThm} can then be restated as $\sigma_s(G)=2$ if and only if $G$ has a nontrivial left-orderable quotient.

A simple argument shows that subsemigroups of torsion groups are in fact subgroups, since the inverse of an element $g$ with finite order is a positive power of $g$.  Therefore, for a torsion group $G$, $\sigma_s(G)=\sigma_g(G)$.  Presently, the author knows of no examples of groups for which the covering number with respect to semigroups is not two nor the covering number with respect to groups.\\
\\
\noindent\textbf{Question 1} Is it true that for every group $G$, either $\sigma_s(G)=2$ or $\sigma_s(G)=\sigma_g(G)$?\\

For instance, one could look for a group $G$ such that $\sigma_s(G)$ is 7 or 11, as 7 and 11 are not equal to $\sigma_g(G)$ for any $G$ \cite{GaronziKappeSwartz18}.

On the other hand, given that $n$ is a covering number of a group $G$ with respect to subsemigroups, we may attempt to give a characterization of groups with covering number $n$  (as we have for two).\\
\\
\noindent\textbf{Question 2} For valid $n>2$, characterize the groups $G$ such that $\sigma_s(G)=n$.

\section*{Acknowledgements}
I would like to sincerely thank Marcin Mazur and Matt Brin for their insight and help with these results.
\bibliographystyle{amsplain}
\bibliography{references_mastercopy}

\providecommand{\bysame}{\leavevmode\hbox to3em{\hrulefill}\thinspace}
\providecommand{\MR}{\relax\ifhmode\unskip\space\fi MR }
\providecommand{\MRhref}[2]{%
  \href{http://www.ams.org/mathscinet-getitem?mr=#1}{#2}
}
\providecommand{\href}[2]{#2}
\begin{thebibliography}{1}

\bibitem{Cohn94}
J.~H.~E. Cohn, \emph{On n-sum groups}, Mathematica Scandinavica \textbf{75}
  (1994), no.~1, 44--58.

\bibitem{DeroinNavasRivas14}
B.~Deroin, A.~Navas, and C.~Rivas, \emph{Groups, orders, and dynamics},
  $\mathrm{arxiv1408.5805}$.

\bibitem{DonovenKappe}
Casey Donoven and Luise-Charlotte Kappe, \emph{Finite coverings of semigroups
  and related structures}, Submitted $\mathrm{arxiv2002.04072}$.

\bibitem{GagolaKappe16}
Stephen~M. Gagola~III and Luise-Charlotte Kappe, \emph{On the covering number
  of loops}, Expositiones Mathematicae \textbf{34} (2016), no.~4, 436 -- 447.

\bibitem{GaronziKappeSwartz18}
Martino Garonzi, Luise-Charlotte Kappe, and Eric Swartz, \emph{On integers that
  are covering numbers of groups}, Experimental Mathematics \textbf{0} (2019),
  no.~0, 1--19.

\bibitem{Lucchini12}
Andrea Lucchini and Attila Mar{\'o}ti, \emph{Rings as the unions of proper
  subrings}, Algebras and Representation Theory \textbf{15} (2012), no.~6,
  1035--1047.

\bibitem{Scorza26}
G.~Scorza, \emph{I gruppi che possone pensarsi come somma di tre lori
  sottogruppi}, Bollettino dell'Unione Matematica Italiana \textbf{5} (1926),
  216--218.

\bibitem{Tomkinson97}
M.~J. Tomkinson, \emph{Groups as the union of proper subgroups.}, Mathematica
  Scandinavica \textbf{81} (1997), 191--198.

\bibitem{Werner15}
Nicholas~J. Werner, \emph{Covering numbers of finite rings}, The American
  Mathematical Monthly \textbf{122} (2015), no.~6, 552--566.

\end{thebibliography}
\end{document}